\documentclass[oneside,english,reqno]{amsart}
\usepackage[T1]{fontenc}
\usepackage[latin9]{inputenc}
\usepackage{babel}

\usepackage{prettyref}
\usepackage{amsthm}
\usepackage{amssymb}
\usepackage{esint}
\usepackage[unicode=true, 
 bookmarks=true,bookmarksnumbered=false,bookmarksopen=false,
 breaklinks=false,pdfborder={0 0 1},backref=false,colorlinks=false]
 {hyperref}
\hypersetup{pdftitle={Fredholm Theory}}
 
\numberwithin{equation}{section} 
\numberwithin{figure}{section} 
\theoremstyle{plain}
\newtheorem{thm}{Theorem}[section]
  \theoremstyle{plain}
  \newtheorem{prop}[thm]{Proposition}
  \theoremstyle{plain}
  \newtheorem{lem}[thm]{Lemma}
  \theoremstyle{remark}
  \newtheorem{rem}[thm]{Remark}
  \theoremstyle{plain}
  \newtheorem{cor}[thm]{Corollary}
 \theoremstyle{definition}
  \newtheorem{example}[thm]{Example}
  \theoremstyle{remark}
  \newtheorem*{acknowledgement*}{Acknowledgement}

\usepackage{lmodern}

\begin{document}

\title[Cosine Transform of Distributions for Complex Minkowski Spaces]{On the Range of Cosine Transform of Distributions for Torus-invariant
Complex Minkowski Spaces }

\author{Yang Liu}

\date{July 19, 2009}
\begin{abstract}
In this paper, we study the ranges of (absolute value) cosine transforms
for which we give a proof for an extended surjectivity theorem by
making applications of the Fredholm's theorem in integral equations,
and show a Hermitian characterization theorem for complex Minkowski
metrics on $\mathbb{C}^{n}$. Moreover, we parametrize the Grassmannian
in an elementary linear algebra approach, and give a characterization
on the image of the (absolute value) cosine transform on the space
of distributions on the Grassmannian $Gr_{2}(\mathbb{C}^{2})$, by
computing the coefficients in the Legendre series expansion of distributions.
\end{abstract}



\maketitle

\section{On the Fredholm Theory in Integral Equations}

Integral equations as different looks from differential equations
appear in mathematical physics and fluid mechanics (see for instance
\cite{Vertlet}) and other fields. A groundbreaking work in the theory
of integral equation was done by Fredholm, \cite{Fredholm}, in 1903.
The following is one of his main theorems on the existence of solutions
to Fredholm integral equations of the second kind
\begin{thm}
\label{thm:frethm}Let $K(x,y)$ and $f(x)$ be real valued functions,
$\lambda\in\mathbb{R}$ and $K(x,y)\in L^{2}([a,b]^{2})$. Then there
exist solutions to the Fredholm integral equation of the second kind
\begin{equation}
\lambda\phi(x)-\int_{a}^{b}K(x,y)\phi(y)\, dy=f(x)\label{eq:integeua}\end{equation}
if and only if $f(x)$ satisfies

\begin{equation}
\int_{a}^{b}\psi(x)f(x)\, dx=0\end{equation}
for any solution $\psi(x)$ to the homogeneous integral equation

\begin{equation}
\lambda\psi(y)-\int_{a}^{b}\psi(y)K(x,y)\, dx=0.\end{equation}

\end{thm}
As for solving integral equations, it is not hard to solve Fredholm
integral equations with separable variables, for that and some other
types of integral equations one can see \cite{A}. One can also use
Fourier on convolution to express solution explicitly if the integral
in \prettyref{eq:integeua} is a convolution.

\section{On $U(1)\times U(1)$-Invariant Complex Finsler Metrics and $U(1)\times U(1)$-Orbits
of $Gr_{2}(\mathbb{C}^{2})$}

Given a complex Finsler space $(\mathbb{C}^{2},F)$, where $F$ is
a complex Finsler metric. One of the main topics in integral geometry
is to find the Crofton measures for Finsler metrics. However, there
is an important class of Finsler metrics, $U(1)\times U(1)$ invariant
complex ones. 

If $F$ is a $U(1)\times U(1)$ invariant complex Finsler metric,
then \begin{equation}
\bar{F}:=F|_{(\mathbb{R}\times\left\{ 0\right\} )\oplus(\mathbb{R}\times\left\{ 0\right\} )}\label{eq:res}\end{equation}
is a Finsler metric on $\mathbb{R}^{2}$. Conversely, one can extend
a Finsler metric on $\mathbb{R}^{2}$ to get a $U(1)\times U(1)$
invariant complex Finsler metric on $\mathbb{C}^{2}$.

For the Crofton measure of $U(1)\times U(1)$ of invariant complex
Finsler metric, we have the following
\begin{thm}[Invariance property of Crofton measure]
The Crofton measure for $U(1)\times U(1)$ invariant complex Finsler
metric $F$ on \textup{$\mathbb{C}^{2}$ is }$U(1)\times U(1)$ invariant. \end{thm}
\begin{proof}
Let $\mu$ be the Crofton measure for the $U(1)\times U(1)$ invariant
complex Finsler metric $F$ on $\mathbb{C}^{2}$ and $d\mu=f(\xi_{1},\xi_{2},\eta)d\xi_{1}d\xi_{2}d\eta$,
then for any $(\bar{z},\bar{w})\in S^{3}$, then we have $F(z,w)=F(e^{i\tilde{\xi}_{1}}z,e^{i\tilde{\xi}_{2}}w)$
for any $(e^{i\tilde{\xi}_{1}},e^{i\tilde{\xi}_{2}})\in U(1)\times U(1)$.

On one hand, \begin{equation}
\begin{array}{lll}
F(z,w) & = & \int_{0}^{2\pi}\int_{0}^{2\pi}\int_{0}^{2\pi}|\cos(\xi_{1}-\bar{\xi}_{1})\cos\eta+\cos(\xi_{2}-\bar{\xi}_{2})\sin\eta|\\
 &  & \qquad\qquad\qquad\qquad\qquad\qquad\qquad\qquad\cdot f(\xi_{1},\xi_{2},\eta)d\xi_{1}d\xi_{2}d\eta\end{array}\end{equation}

On the other hand, we know for any $\mbox{(\ensuremath{\bar{z}},\ensuremath{\bar{w}})=(\ensuremath{e^{i\bar{\xi}_{1}}\cos\bar{\eta}},\ensuremath{e^{i\bar{\xi}_{2}}\sin\bar{\eta}})\ensuremath{\in S^{3}}}$,\begin{equation}
\begin{array}{lll}
F(\bar{z},\bar{w}) & = & \int_{0}^{2\pi}\int_{0}^{2\pi}\int_{0}^{2\pi}|\cos(\xi_{1}-\bar{\xi}_{1}-\tilde{\xi}_{1})\cos\eta+\cos(\xi_{2}-\bar{\xi}_{2}-\tilde{\xi}_{2})\sin\eta|\\
 &  & \qquad\qquad\qquad\qquad\qquad\qquad\qquad\qquad\cdot f(\xi_{1},\xi_{2},\eta)d\xi_{1}d\xi_{2}d\eta\\
 & = & \int_{0}^{2\pi}\int_{0}^{2\pi}\int_{0}^{2\pi}|\cos(\xi_{1}-\bar{\xi}_{1})\cos\eta+\cos(\xi_{2}-\bar{\xi}_{2}-\tilde{\xi}_{2})\sin\eta|\\
 &  & \qquad\qquad\qquad\qquad\qquad\qquad\qquad\cdot f(\xi_{1}+\tilde{\xi}_{1},\xi_{2}+\tilde{\xi}_{2},\eta)d\xi_{1}d\xi_{2}d\eta\end{array}\end{equation}
by change of variables. Using the injectivity theorem of cosine transform,
Proposition 3.4.12 in \cite{G}, from $F(z,w)=F(e^{i\tilde{\xi}_{1}}z,e^{i\tilde{\xi}_{2}}w)$
we have \begin{equation}
f(\xi_{1},\xi_{2},\eta)=f(\xi_{1}+\tilde{\xi}_{1},\xi_{2}+\tilde{\xi}_{2},\eta)\end{equation}
for any $\tilde{\xi}_{1},\tilde{\xi}_{2}\in[0,2\pi]$.
\end{proof}
Since the function $f$ is independent of $\xi_{1}$ and $\xi_{2}$
by the invariance of the complex norm under $U(1)\times U(1)$ action,
so it can be denoted as $f(\eta)$.

In the next, we consider the action of torus $U(1)\times U(1)$ on
the space of real $2$-planes in the complex plane, $Gr_{2}(\mathbb{C}^{2})$.
The following proposition about the orbits of torus action was proposed
by Joe Fu, but here we provide a proof with linear algebra flavor 
\begin{prop}[Orbits parametrization of the Grassmannian]
\label{pro:The-orbit-of}The orbits of $Gr_{2}(\mathbb{C}^{2})$
acted by torus actions can be parametrized as \begin{equation}
\left\{ span_{\mathbb{R}}((\cos\psi,\sin\psi),(\sqrt{-1}\cos(\theta+\psi),\sqrt{-1}\sin(\theta+\psi))):(\theta,\psi)\in[0,\frac{\pi}{2}]^{2}\right\} .\end{equation}

\end{prop}
Since a torus action preserves the argument differences of each component
of any two vectors in $\mathbb{C}^{2}$, so to prove Proposition \prettyref{pro:The-orbit-of},
it suffices to show the following
\begin{lem}
\label{lem:For-any-planeexistence}For any plane $P\in Gr_{2}(\mathbb{C}^{2})$,
either there exist some $(z_{0},w_{0})\in P\setminus\left\{ 0\right\} $
and $r,s\in\mathbb{R}$ such that $(\sqrt{-1}rz_{0},\sqrt{-1}sw_{0})\in P$,
in other words, \begin{equation}
P=span_{\mathbb{R}}((z_{0},w_{0}),(\sqrt{-1}rz_{0},\sqrt{-1}sw_{0})),\end{equation}
or there exists a pair of vectors $(z_{1},w_{1}),(z_{2},w_{2})\in P\setminus\left\{ 0\right\} $
such that $z_{1}w_{2}=z_{2}w_{1}=0$.\end{lem}
\begin{rem}
We call the vector $(z_{0},w_{0})$ a quasi-$J$-characteristic vector
of the plane $P$ . In particular, every non-zero vector in a complex
line $L$ in $\mathbb{C}^{2}$ is a quasi-$J$-characteristic vector
of $L$.
\end{rem}
Let $\mathbf{\mathbf{\mathbb{T}^{2}}}:=\left\{ span_{\mathbb{R}}((z,0),(0,w)):z,w\in U(1)\right\} \cong U(1)\times U(1)$,
then in fact the latter part of the conclusion in \prettyref{lem:For-any-planeexistence}
is derived from the planes in $\mathbf{\mathbf{\mathbb{T}^{2}}}$.
For planes which are not in $\mathbf{\mathbf{\mathbb{T}^{2}}}$, we
need to show that they generate the former part of the conclusion
in \prettyref{lem:For-any-planeexistence}, which is geometrically
equivalent to
\begin{lem}[Intersection lemma]
For any $P\in Gr_{2}(\mathbb{C}^{2})\setminus\mathbf{\mathbf{\mathbb{T}^{2}}}$,
there exist $r,s\in\mathbb{R}$ such that $dim(P_{r,s}\cap P)>0$
where $P_{r,s}:=\left\{ (\sqrt{-1}rz,\sqrt{-1}sw):(z,w)\in P\right\} $.\end{lem}
\begin{proof}
Let $P=span_{\mathbb{R}}((z_{1},w_{1}),(z_{2},w_{2}))\in Gr_{2}(\mathbb{C}^{2})\setminus\mathbf{\mathbf{\mathbb{T}^{2}}}$,
and $z_{i}=x_{i}+\sqrt{-1}y_{i}$ and $w_{i}=u_{i}+\sqrt{-1}v_{i}$
for $i=1,2$. Using the determinants of block matrices by partitioning
a matrix, one can obtain that \begin{equation}
det\left(\begin{array}{cccc}
x_{1} & y_{1} & u_{1} & v_{1}\\
x_{2} & y_{2} & u_{2} & v_{2}\\
-ry_{1} & rx_{1} & -sv_{1} & su_{1}\\
-ry_{2} & rx_{2} & -sv_{2} & su_{2}\end{array}\right)=Ar^{2}+Brs+Cs^{2}\label{eq:det}\end{equation}
for some $A,B,C\in\mathbb{R}$ with $A=-C=det(M_{11})det(M_{22})$
where $M_{11}=\left(\begin{array}{cc}
x_{1} & y_{1}\\
x_{2} & y_{2}\end{array}\right)$ and $M_{22}=\left(\begin{array}{cc}
-v_{1} & u_{1}\\
-v_{2} & u_{2}\end{array}\right)$. Therefore, there exist $r,s\in\mathbb{R}$ and either $r$ or $s$
is not $0$, such that the determinant \prettyref{eq:det} is identical
to $0$. It follows that \begin{equation}
P\oplus P_{r,s}=span_{\mathbb{R}}((z_{1},w_{1}),(z_{2},w_{2}),(\sqrt{-1}rz_{1},\sqrt{-1}sw_{1}),(\sqrt{-1}rz_{2},\sqrt{-1}sw_{2}))\subsetneq\mathbb{C}^{2},\end{equation}
and then we have $dim(P_{r,s}\cap P)>0$ by the inclusion-exclusion
principle.
\end{proof}
Thus we have shown \prettyref{lem:For-any-planeexistence}. Furthermore,
one can choose appropriate $(\theta,\psi)\in[0,\frac{\pi}{2}]^{2}$
such that $span_{\mathbb{R}}((\cos\psi,\sin\psi),(\sqrt{-1}\cos(\theta+\psi),\sqrt{-1}\sin(\theta+\psi)))$
and $P$ are on the same orbit of $Gr_{2}(\mathbb{C}^{2})$ acted
by torus actions. So we have finished the proof for Proposition \prettyref{pro:The-orbit-of}.

\section{A Surjectivity Theorem}

We want to extend the surjectivity theorem on the cosine transform
to functions which are not even differentiable away from zero by making
applications of Fredholm's theorem on integral equations.
\begin{thm}[Surjectivity theorem ]
For any $U(1)\times U(1)$-invariant function $F:\mathbb{C}^{2}\rightarrow\mathbb{R}$
with homogeneity of magnitude, there is some function $f$ on $S^{n}$,
such that \begin{equation}
F(\cdot)=\int_{S^{3}}|\langle\xi,\cdot\rangle|f(\xi)d\xi\end{equation}
\end{thm}
\begin{proof}
Let \begin{equation}
K(\eta,\bar{\eta}):=\int_{0}^{2\pi}\int_{0}^{2\pi}|\cos(\xi_{1}-\bar{\xi}_{1})\cos\eta\cos\bar{\eta}+\cos(\xi_{2}-\bar{\xi}_{2})\sin\eta\sin\bar{\eta}|d\xi_{1}d\xi_{2}\end{equation}
because the double integral is independent of $\bar{\xi}_{1}$ and
$\bar{\xi}_{2}$. Considering the integral equation \begin{equation}
\int_{0}^{2\pi}K(\eta,\bar{\eta})f(\eta)\, d\eta=F(\eta),\label{eq:torusexist}\end{equation}
and applying \prettyref{thm:frethm} to it, we know that there exists
some $f(\eta)$ satisfying integral equation \prettyref{eq:torusexist}.
\end{proof}
In the theory of convex bodies, \cite{S}, the support function of
the unit ball in a Minkowski space is actually the metric function,
and the ball is called a generalized zonoid if its support function
is in the range of cosine transform on the functions on $S^{3}$.
Hence we have the following
\begin{cor}
The unit ball of any complex Minkowski plane $(\mathbb{C}^{2},F)$
with $U(1)\times U(1)$-invariant complex Minkowski metric $F$ is
a generalized zonoid.\end{cor}
\begin{rem}
To apply the integral equation theory, one does not need any smoothness
condition on the metric. However, the approach of integral equation
theory can not be generalized to Minkowski metric on $\mathbb{R}^{n}$
for any $n$, in which the unit ball could be not a generalized zonoid,
for example, the octahedron, as pointed out by Joe Fu, in $\mathbb{R}^{3}$
with $l^{1}$ metric.
\end{rem}

\section{On a Complex Minkowski Metric To Be Hermitian on $\mathbb{C}^{n}$ }

From the perspective of complex integral geometry, the following theorem
on a characterization of complex Minkowski metric $\mathbb{C}^{n}$
to be Hermitian is established
\begin{thm}[Characterization of Hermitian metric]
\label{thm:Suppose-that-thm}Suppose that $(\mathbb{C}^{n},F)$ is
a complex Minkowski space. Then the Holmes-Thompson valuation, that
is extended from the Holmes-Thompson area on $(\mathbb{C}^{n},F)$,
restricted on $\mathbb{C}\mathbb{P}^{n-1}$ is in the range of the
cosine transform on $C(\mathbb{C}\mathbb{P}^{n-1})$ if and only if
the complex Minkowski metric $F$ is Hermitian.\end{thm}
\begin{proof}
For any fixed complex line $L\in\mathbb{C}\mathbb{P}^{n-1}$, let
$U$ be the rectangle spanned by $v:=(z_{1},\cdots,z_{n})\in L$ and
$\sqrt{-1}v\in L$. Since $F$ is $U(1)$-invariant on $L$, then
the Holmes-Thompson area of $U$ is\begin{equation}
HT^{2}(U)=F^{2}(v).\end{equation}

On the other hand, for any complex line $\tilde{L}:=span_{\mathbb{C}}(\tilde{e})\in\mathbb{C}\mathbb{P}^{n-1}$
where $\tilde{e}:=(\tilde{z}_{1},\cdots,\tilde{z}_{n})\in\mathbb{C}^{n}$
with $|\tilde{e}|=(\sum_{i=1}^{n}|\tilde{z}_{i}|^{2})^{1/2}=1$, we
know that \begin{equation}
\begin{array}{lcl}
area(\pi_{\tilde{L}}(U)) & = & \mid det\left(\begin{array}{cc}
Re(\langle v,\tilde{e}\rangle_{\mathbb{C}}) & Im(\langle v,\tilde{e}\rangle_{\mathbb{C}})\\
-Im(\langle v,\tilde{e}\rangle_{\mathbb{C}}) & Re(\langle v,\tilde{e}\rangle_{\mathbb{C}})\end{array}\right)\mid\\
 & = & |\langle v,\tilde{e}\rangle_{\mathbb{C}}|^{2},\end{array}\label{eq:loong}\end{equation}
in which $\langle v,\tilde{e}\rangle_{\mathbb{C}}$ is the complex
inner product, and $area(\pi_{\tilde{L}}(U))$ is independent of the
choice of unit vector $\tilde{e}$ in $\tilde{L}$.

If $HT^{2}$ is in the range of cosine transform on $C(\mathbb{C}\mathbb{P}^{n-1})$,
then there exists some function $f:\mathbb{C}\mathbb{P}^{n-1}\rightarrow\mathbb{R}$,
such that 

\begin{equation}
\int_{\mathbb{C}\mathbb{P}^{n-1}}area(\pi_{\tilde{L}}(U))f(\tilde{L})d\tilde{L}=F^{2}(v).\label{eq:2}\end{equation}
Since $\mathbb{C}\mathbb{P}^{n-1}=S^{2n-1}/U(1)$, then by \prettyref{eq:loong}
we have\begin{equation}
\begin{array}{lll}
\int_{\mathbb{C}\mathbb{P}^{n-1}}area(\pi_{\tilde{L}}(U))f(\tilde{L})d\tilde{L} & = & \int_{S^{2n-1}/U(1)}|\langle v,\tilde{e}\rangle_{\mathbb{C}}|^{2}f(\tilde{e})d\tilde{e}\end{array}.\end{equation}
Written in terms of components of $\tilde{e}$ and $v$, \begin{equation}
\begin{array}{lll}
\int_{\mathbb{C}\mathbb{P}^{n-1}}area(\pi_{\tilde{L}}(U))f(\tilde{L})d\tilde{L} & = & \int_{S^{2n-1}/U(1)}\sum_{i=1}^{n}|z_{i}\bar{\tilde{z}}_{i}|^{2}f(\tilde{e})d\tilde{e}\\
 & = & \int_{S^{2n-1}/U(1)}\sum_{i,j=1}^{n}z_{i}\bar{\tilde{z}}_{i}\bar{z}_{j}\tilde{z}_{j}f(\tilde{e})d\tilde{e}\\
 & = & \sum_{i,j=1}^{n}z_{i}\bar{z}_{j}\int_{S^{2n-1}/U(1)}\bar{\tilde{z}}_{i}\tilde{z}_{j}f(\tilde{e})d\tilde{e}.\end{array}\end{equation}
Let $h_{i\bar{j}}:=\int_{S^{2n-1}/U(1)}\bar{\tilde{z}}_{i}\tilde{z}_{j}f(\tilde{e})d\tilde{e}$,
then $h_{i\bar{j}}=\bar{h}_{j\bar{i}}$ since $f(\tilde{e})\in\mathbb{R}$.
Thus it follows from \prettyref{eq:2} that $F^{2}(v)=\sum_{i,j=1}^{n}h_{i\bar{j}}z_{i}\bar{z}_{j}$
is Hermitian. 

Conversely, one can see, by the fact that $\left\{ \bar{\tilde{z}}_{i}\tilde{z}_{j}:i,j=1,\cdots n\right\} $
are linearly independent in the Hilbert space $L^{2}(S^{2n-1}/U(1))$
and \begin{equation}
F|_{\mathbb{C}\mathbb{P}^{n-1}}\in C(\mathbb{C}\mathbb{P}^{n-1})\subset L^{2}(S^{2n-1}/U(1)),\end{equation}
that if $F$ is Hermitian then the Holmes-Thompson valuation restricted
on $\mathbb{C}\mathbb{P}^{n-1}$ is in the range of the cosine transform
on $C(\mathbb{C}\mathbb{P}^{n-1})$ .\end{proof}
\begin{rem}
Form the proof of \prettyref{thm:Suppose-that-thm}, we know that
the range of the cosine transform on $C(\mathbb{C}\mathbb{P}^{n-1})$
is finite dimensional.
\end{rem}

\section{Revolutions of Spheres and Torus Actions}

We have shown the following
\begin{prop}
\label{pro:The-orbit-of-1}The orbits of $Gr_{2}(\mathbb{C}^{2})$
acted by torus actions can be parametrized as \begin{equation}
\left\{ span_{\mathbb{R}}((\cos\psi,\sin\psi),(\sqrt{-1}\cos(\theta+\psi),\sqrt{-1}\sin(\theta+\psi))):(\theta,\psi)\in[0,\frac{\pi}{2}]^{2}\right\} .\label{eq:orbit}\end{equation}

\end{prop}
Gluck and Warner in \cite{GW} gave an isomorphism from $Gr_{2}^{+}(\mathbb{R}^{4})$
to $S^{2}\times S^{2}$ in \prettyref{eq:iso}, that is expressed
explicitly by\begin{equation}
\begin{array}{c}
\iota(v_{1}\wedge v_{2}):=(\frac{\sqrt{2}}{2}(v_{1}\wedge v_{2}+(v_{1}\wedge v_{2})^{\bot}),\frac{\sqrt{2}}{2}(v_{1}\wedge v_{2}-(v_{1}\wedge v_{2})^{\bot})),\end{array}\label{eq:map}\end{equation}
in which $\left\{ v_{1},v_{2}\right\} $ is an orthonormal basis of
the plane spanned by them in $Gr_{2}^{+}(\mathbb{R}^{4})$ and $(v_{1}\wedge v_{2})^{\bot}$
here denotes the wedge of the orthonormal basis of the complement
of $v_{1}\wedge v_{2}$, and so we have the Cartesian product decomposition
\begin{equation}
Gr_{2}^{+}(\mathbb{R}^{4})\cong S^{2}\times S^{2}.\label{eq:iso}\end{equation}

In \cite{GH}, Goodey and Howard described the bases \begin{equation}
b_{1}^{+}=\frac{\sqrt{2}}{2}(e_{1}\wedge e_{2}+e_{3}\wedge e_{4}),\, b_{2}^{+}=\frac{\sqrt{2}}{2}(e_{1}\wedge e_{3}-e_{2}\wedge e_{4}),\, b_{3}^{+}=\frac{\sqrt{2}}{2}(e_{1}\wedge e_{4}+e_{2}\wedge e_{3}),\prettyref{eq:1c}\label{eq:1c}\end{equation}
and \begin{equation}
b_{1}^{-}=\frac{\sqrt{2}}{2}(e_{1}\wedge e_{2}-e_{3}\wedge e_{4}),\, b_{2}^{-}=\frac{\sqrt{2}}{2}(e_{1}\wedge e_{3}+e_{2}\wedge e_{4}),\, b_{3}^{-}=\frac{\sqrt{2}}{2}(e_{1}\wedge e_{4}-e_{2}\wedge e_{3}),\prettyref{eq:2c}\label{eq:2c}\end{equation}
where $\left\{ e_{1},e_{2},e_{3},e_{4}\right\} $ is the basis for
$\mathbb{R}^{4}$, for the components $\wedge_{+}^{2}(\mathbb{R}^{4})$
and $\wedge_{-}^{2}(\mathbb{R}^{4})$ respectively in the vector space
decomposition \begin{equation}
\wedge^{2}(\mathbb{R}^{4})=\wedge_{+}^{2}(\mathbb{R}^{4})\oplus\wedge_{-}^{2}(\mathbb{R}^{4}).\end{equation}

In $\mathbb{C}^{2}$, we set $\left\{ e_{1},e_{2},e_{3},e_{4}\right\} $
to be a basis of $\mathbb{C}^{2}$ such that $(z,w)=Re(z)e_{1}+Im(z)e_{2}+Re(w)e_{3}+Im(w)e_{4}$
for any $(z,w)\in\mathbb{C}^{2}$. Thus, we can identify the orbit
space \prettyref{eq:orbit} with a subspace of $S^{2}\times S^{2}$,
and it turns out that
\begin{lem}
The orbit space \prettyref{eq:orbit} is in the quotient space of
the Cartesian product $S^{1}\times S^{1}$ by identifying antipodal
points, where the circles $S^{1}$ are the equators of the spheres
$S^{2}$ with $b_{2}^{+}$ and $b_{2}^{-}$ as north poles in \prettyref{eq:iso}.
\begin{proof}
Let \begin{equation}
P_{\theta,\psi}:=span_{\mathbb{R}}((\cos\psi,\sin\psi),(\sqrt{-1}\cos(\theta+\psi),\sqrt{-1}\sin(\theta+\psi)))\label{eq:represe}\end{equation}
in the orbit space \prettyref{eq:orbit}. Then from \prettyref{eq:map}
the first component of $\iota(P_{\theta,\psi})$, denoted by $\iota_{1}(P_{\theta,\psi})$,
is \begin{equation}
\begin{array}{lll}
\iota_{1}(P_{\theta,\psi}) & = & \frac{\sqrt{2}}{2}(\cos\psi\, e_{1}+\sin\psi\, e_{3})\wedge(\cos(\theta+\psi)\, e_{2}+\sin(\theta+\psi)\, e_{4})\\
 &  & \:+\frac{\sqrt{2}}{2}(\sin\psi\, e_{1}-\cos\psi\, e_{3})\wedge(\sin(\theta+\psi)\, e_{2}-\cos(\theta+\psi)\, e_{4})\\
 & = & \frac{\sqrt{2}}{2}\cos\theta(e_{1}\wedge e_{2}+e_{3}\wedge e_{4})+\frac{\sqrt{2}}{2}\sin\theta(e_{1}\wedge e_{4}+e_{2}\wedge e_{3})\\
 & = & \cos\theta\, b_{1}^{+}+\sin\theta\, b_{3}^{+},\end{array}\end{equation}
and analogously the second component of $\iota(P_{\theta,\psi})$,
denoted by $\iota_{2}(P_{\theta,\psi})$, is \begin{equation}
\iota_{2}(P_{\theta,\psi})=\cos(2\psi+\theta)\, b_{1}^{-}+\sin(2\psi+\theta)\, b_{3}^{-}.\end{equation}
Hence $\iota_{1}(P_{\theta,\psi})$ and $\iota_{2}(P_{\theta,\psi})$
don't have $b_{2}^{+}$ or $b_{2}^{-}$ component, then the claim
follows.
\end{proof}
\end{lem}
When acting $2$-planes in $Gr_{2}(\mathbb{C}^{2})$ by torus actions,
we have 
\begin{lem}
The torus action of a $2$-plane in $Gr_{2}(\mathbb{C}^{2})$ is a
revolution along the axes $b_{1}^{+}$ and $b_{1}^{-}$ on the component
spheres $S^{2}$ in \prettyref{eq:iso}.\end{lem}
\begin{proof}
Let $P_{\theta,\psi}^{\alpha,\beta}$ be the $2$-plane obtained from
a torus action $(e^{\sqrt{-1}\alpha},e^{\sqrt{-1}\beta})$ on the
$2$-plane $P_{\theta,\psi}$, \prettyref{eq:represe}, then \begin{equation}
P_{\theta,\psi}^{\alpha,\beta}=span_{\mathbb{R}}((\cos\psi\, e^{\sqrt{-1}\alpha},\sin\psi\, e^{\sqrt{-1}\beta}),\sqrt{-1}(e^{\sqrt{-1}\alpha}\cos(\theta+\psi),e^{\sqrt{-1}\beta}\sin(\theta+\psi))).\end{equation}
and it orthogonal complement plane $(P_{\theta,\psi}^{\alpha,\beta})^{\bot}$
can be expressed as\begin{equation}
span_{\mathbb{R}}((\sin\psi\, e^{\sqrt{-1}\alpha},-\cos\psi\, e^{\sqrt{-1}\beta}),(e^{\sqrt{-1}(\alpha+\frac{\pi}{2})}\sin(\theta+\psi),e^{\sqrt{-1}(\beta-\frac{\pi}{2})}\cos(\theta+\psi))).\end{equation}

Therefore, applying the map \prettyref{eq:map} to $P_{\theta,\psi}^{\alpha,\beta}$,
we get the first component of $\iota(P_{\theta,\psi}^{\alpha,\beta})$,

\begin{equation}
\begin{array}{lll}
\iota_{1}(P_{\theta,\psi}^{\alpha,\beta}) & = & \frac{\sqrt{2}}{2}(\cos\theta\, e_{1}\wedge e_{2}-\sin\theta\sin(\alpha+\beta)\, e_{1}\wedge e_{3}+\sin\theta\cos(\alpha+\beta)\, e_{1}\wedge e_{4}\\
 &  & \,\,+\sin\theta\cos(\alpha+\beta)\, e_{2}\wedge e_{3}+\sin\theta\sin(\alpha+\beta)\, e_{2}\wedge e_{4}+\cos\theta\, e_{3}\wedge e_{4})\\
 & = & \cos\theta\, b_{1}^{+}-\sin\theta\sin(\alpha+\beta)\, b_{2}^{+}+\sin\theta\cos(\alpha+\beta)\, b_{3}^{+},\end{array}\label{eq:1134}\end{equation}
and the second component of $\iota(P_{\theta,\psi}^{\alpha,\beta})$,\begin{equation}
\begin{array}{lll}
\iota_{2}(P_{\theta,\psi}^{\alpha,\beta}) & = & \cos(2\psi+\theta)\, b_{1}^{-}-\sin(2\psi+\theta)\sin(\alpha-\beta)\, b_{2}^{-}\\
 &  & \quad+\sin(2\psi+\theta)\cos(\alpha-\beta)\, b_{3}^{-}.\end{array}\label{eq:22twet}\end{equation}
Thus we can see that the claim follows from the expressions of $\iota_{1}(P_{\theta,\psi}^{\alpha,\beta})$
and $\iota_{2}(P_{\theta,\psi}^{\alpha,\beta})$ in \prettyref{eq:1134}
and \prettyref{eq:22twet}.
\end{proof}
One can parametrize the two spheres by\begin{equation}
\xi(x,\phi_{1}):=(\sqrt{1-x^{2}}\cos\phi_{1},\sqrt{1-x^{2}}\sin\phi_{1},x)\label{eq:p1}\end{equation}
and \begin{equation}
\eta(y,\phi_{2}):=(\sqrt{1-y^{2}}\cos\phi_{2},\sqrt{1-y^{2}}\sin\phi_{2},y).\label{eq:p2}\end{equation}
In \cite{GH} , Goodey and Howard characterized the kernel of general
cosine transform \begin{equation}
\mathcal{C}(f)(P):=\int_{Q\in Gr_{2}(\mathbb{R}^{4})}|\langle P,Q\rangle|f(Q)dQ\label{eq:cosi}\end{equation}
by Legendre polynomials \begin{equation}
p_{n}(x)=\frac{1}{2^{n}n!}\frac{d^{n}}{dx^{n}}(x^{2}-1)^{n}.\end{equation}
However, if $f$ is a $\mathcal{U}(1)\times\mathcal{U}(1)$-invariant
function on $Gr_{2}(\mathbb{C}^{2})$, then it is independent of $\phi_{1}$
and $\phi_{2}$. We have 
\begin{thm}
\label{thm:The-kernel-of}The kernel of cosine transform\begin{equation}
\mathcal{C_{T}}(f)(P):=\int_{Q\in Gr_{2}(\mathbb{C}^{2})}|\langle P,Q\rangle|f(Q)dQ,\label{eq:ct}\end{equation}
in which $f\in L^{2}(Gr_{2}(\mathbb{C}^{2}))$ is torus invariant,
is \begin{equation}
\left\{ \sum_{\begin{array}{c}
m-n\,\mbox{even,}|m-n|>2\end{array}}c_{m,n\,}p_{m}(\cos\theta)p_{n}(\cos(2\psi+\theta))\right\} ,\label{eq:kkk}\end{equation}
in which $(\theta,\psi)$\textup{ are the parameters of the orbit}
in \prettyref{eq:orbit}\textup{ that $Q$ belongs to}.\end{thm}
\begin{proof}
From the expressions \prettyref{eq:1134} and \prettyref{eq:22twet},
we know that $x=\cos\theta$ and $y=\cos(2\psi+\theta)$ in \prettyref{eq:p1}
and \prettyref{eq:p2}. On the other hand, we also know that the kernel
of cosine transform \prettyref{eq:ct-ne} is exactly the sum of functions
$p_{m}(x)p_{n}(y)$, the product of Legendre polynomials, for $m-n$
even and $|m-n|\geq3$. So the claim follows.
\end{proof}
From \prettyref{thm:The-kernel-of}, we know that the cosine transform
annihilates the pieces of $p_{m}(\sin\theta)p_{n}(\sin(2\psi+\theta))$,
$|m-n|\geq3$, in the Legendre series expansions of functions in $L^{2}(Gr_{2}(\mathbb{C}^{2}))$,
and it follows from from the self-adjoint property of the cosine tranaform
operator that
\begin{cor}
\label{cor:The-image-of}The image of cosine transform\begin{equation}
\mathcal{C_{T}}(f)(P):=\int_{Q\in Gr_{2}(\mathbb{C}^{2})}|\langle P,Q\rangle|f(Q)dQ,\end{equation}
in which $f\in L^{2}(Gr_{2}(\mathbb{C}^{2}))$ is torus invariant,
is the space of torus invariant continuous functions on $Gr_{2}(\mathbb{C}^{2})$
which are not in \prettyref{eq:kkk}, in other words, \begin{equation}
im(\mathcal{C_{T}})=\left\{ \sum_{\begin{array}{c}
|m-n|=0\,\mbox{or}\,2\end{array}}c_{m,n\,}p_{m}(\cos\theta)p_{n}(\cos(2\psi+\theta))\right\} .\end{equation}

\end{cor}

\section{On the Volume of a Convex Body of Mokowski sum }

Now we consider $HT_{4}(K+\sqrt{-1}K)$ for a convex body $K$ in
a plane $P$ in $Gr_{2}(\mathbb{C}^{2})$. First, we know that there
is a function $f$ on $Gr_{2}(\mathbb{C}^{2})$ such that \begin{equation}
HT_{4}(K+\sqrt{-1}K)=f(P)HT_{2}(K)^{2},\label{eq:fun}\end{equation}
for any convex body $K\subset P\in Gr_{2}(\mathbb{C}^{2})$. Furthermore,
$f(P)=0$ for all $P\in\mathbb{C}\mathbb{P}^{1}$. In general, by
the property of Euclidean volumes, we know that 

\begin{equation}
f(P)=\frac{c\sin^{2}\theta}{(Kl_{HT^{2}}(P))^{2}}\label{eq:func}\end{equation}
where $c=\frac{HT^{4}(B_{F}^{4})}{vol_{4}(B_{F}^{4})}$ for the unit
ball $B_{F}^{4}$ of $(\mathbb{C}^{2},F)$ is a constant with respect
to $P$ and $Kl_{HT^{2}}$ is the Klain function of the Holmes-Thompson
area, $HT^{2}$. 

In the next, we are going to compute the cosine transform to get the
Klain function of the Holmes-Thompson area, $HT^{2}$. First, the
projection area of a square $E$ spanned by $(\cos\psi_{0},\sin\psi_{0})$
and $(\sqrt{-1}\cos(\theta_{0}+\psi_{0}),\sqrt{-1}\sin(\theta_{0}+\psi_{0}))$
from one plane, \begin{equation}
P_{\theta_{0},\psi_{0}}:=span_{\mathbb{R}}((\cos\psi_{0},\sin\psi_{0}),(\sqrt{-1}\cos(\theta_{0}+\psi_{0}),\sqrt{-1}\sin(\theta_{0}+\psi_{0})))\end{equation}
in \prettyref{eq:orbit} to another arbitrary plane\begin{equation}
P_{\theta,\psi}^{\alpha,\beta}:=span_{\mathbb{R}}((e^{i\alpha}\cos\psi,e^{i\beta}\sin\psi),(\sqrt{-1}e^{i\alpha}\cos(\theta+\psi),\sqrt{-1}e^{i\beta}\sin(\theta+\psi)))\end{equation}
is\begin{equation}
\begin{array}{lll}
\pi_{P_{\theta,\psi}^{\alpha,\beta}}(E) & := & |\langle(\cos\psi_{0},\sin\psi_{0}),(e^{i\alpha}\cos\psi,e^{i\beta}\sin\psi)\rangle_{\mathbb{R}}\\
 &  & \,\,\,\cdot\langle\sqrt{-1}(\cos\phi_{0},\sin\phi_{0}),\sqrt{-1}(e^{i\alpha}\cos\phi,e^{i\beta}\sin\phi)\rangle_{\mathbb{R}}\\
 &  & \,\,\,\,-\langle\sqrt{-1}(\cos\phi_{0},\sin\phi_{0}),(e^{i\alpha}\cos\psi,e^{i\beta}\sin\psi)\rangle_{\mathbb{R}}\\
 &  & \,\,\,\,\,\,\,\cdot\langle(\cos\psi_{0},\sin\psi_{0}),\sqrt{-1}(e^{i\alpha}\cos\phi,e^{i\beta}\sin\phi)\rangle_{\mathbb{R}}|\end{array}\end{equation}
for $\phi_{0}:=\theta_{0}+\psi_{0}$ and $\phi:=\theta+\psi$. By
spherical harmonics, there exists a Crofton measure for any complex
Finsler metric under some smoothness condition, and we have shown
the following
\begin{prop}
The Crofton measure for $\mathcal{U}(1)\times\mathcal{U}(1)$ invariant
complex Finsler metric $F$ on \textup{$\mathbb{C}^{2}$ is }$\mathcal{U}(1)\times\mathcal{U}(1)$
invariant. 
\end{prop}
Furthermore, one can obtain $\mathcal{U}(1)\times\mathcal{U}(1)$
invariant Crofton measure for $HT^{2}$ from the one for the metric.

\section{The Cosine Transform of the Complex $l^{1}$ Space's Singular Measure }

Let us consider the case of complex $l^{1}$ for the Kain function
of $HT^{2}$. As we know, the Crofton measure $\mu$ for $HT^{2}$
is induced from the intersection map \begin{equation}
\pi:Gr_{3}(\mathbb{R}^{4})\times Gr_{3}(\mathbb{R}^{4})\setminus\triangle\rightarrow Gr_{2}(\mathbb{R}^{4}),\label{eq:inter}\end{equation}
and indeed \begin{equation}
\mu=\delta_{\left\{ \mathbb{C}\times\left\{ 0\right\} \right\} }+\delta_{\left\{ \left\{ 0\right\} \times\mathbb{C}\right\} }+\frac{1}{4\pi}\lambda_{\mathbb{T}}\end{equation}
where $\lambda_{T}$ is the uniform measure on \begin{equation}
\mathbb{T}:=\left\{ span(v,w):v\in\mathbb{C}\times\left\{ 0\right\} ,w\in\left\{ 0\right\} \times\mathbb{C}\right\} .\label{eq:tor}\end{equation}
Applying the cosine transform,\begin{equation}
\begin{array}{lll}
\mathcal{C_{T}}(f)(P) & = & \int_{Q\in Gr_{2}(\mathbb{C}^{2})}|\langle P,Q\rangle|d\mu(Q)\\
 & = & |\langle P,\mathbb{C}\times\left\{ 0\right\} \rangle|+|\langle P,\left\{ 0\right\} \times\mathbb{C}\rangle|+\frac{1}{4\pi}\int_{Q\in\mathbb{T}}|\langle P,Q\rangle|d\lambda_{T}(Q).\end{array}\label{eq:ct-ne}\end{equation}
Let \begin{equation}
P=span_{\mathbb{R}}((\cos\psi,\sin\psi),(\sqrt{-1}\cos(\theta+\psi),\sqrt{-1}\sin(\theta+\psi)))\end{equation}
and \begin{equation}
Q=span_{\mathbb{R}}(e^{\sqrt{-1}x},0),(0,e^{\sqrt{-1}y}),\end{equation}
then \begin{equation}
|\langle P,\mathbb{C}\times\left\{ 0\right\} \rangle|=|\cos(\theta+\psi)\cos\psi|,\end{equation}
\begin{equation}
|\langle P,\left\{ 0\right\} \times\mathbb{C}\rangle|=|\sin(\theta+\psi)\sin\psi|\end{equation}
and\begin{equation}
\begin{array}{lll}
I:=\int_{Q\in\mathbb{T}}|\langle P,Q\rangle|d\lambda_{\mathbb{T}}(Q) & = & \int_{0}^{2\pi}\int_{0}^{2\pi}|\cos\psi\sin(\theta+\psi)\cos x\sin y\\
 &  & \quad+\sin\psi\cos(\theta+\psi)\sin x\cos y|dxdy.\end{array}\label{eq:di}\end{equation}
Thus, we have\begin{equation}
\begin{array}{lll}
\mathcal{C_{T}}(f)(P) & = & |\cos(\theta+\psi)\cos\psi|+|\sin(\theta+\psi)\sin\psi|\\
 &  & \,+\frac{1}{4\pi}\int_{0}^{2\pi}\int_{0}^{2\pi}|\cos\psi\sin(\theta+\psi)\cos x\sin y\\
 &  & \qquad\quad\qquad\quad+\sin\psi\cos(\theta+\psi)\sin x\cos y|dxdy.\end{array}\label{eq:cosin}\end{equation}

\begin{example}
If $P=\mathbb{C}\times\left\{ 0\right\} $ or $\left\{ 0\right\} \times\mathbb{C}$,
then we know that $P$ is a Euclidean $2$-plane, so the Klain function
of $HT^{2}$ at $P$ is $1$. On the other hand, we also get $\mathcal{C_{T}}(f)(P)=1$
from the right hand side of \prettyref{eq:ct-ne}; If $P$ is a complex
line, $span_{\mathbb{R}}((\cos\psi,\sin\psi),\sqrt{-1}(\cos\psi,\sin\psi)))$,
then $P$ is still a Euclidean $2$-plane. By \prettyref{eq:cosin}
we have \begin{equation}
\mathcal{C_{T}}(f)(P)=\cos^{2}\psi+\sin^{2}\psi+2|\cos\psi\sin\psi|=(|\cos\psi|+|\sin\psi|)^{2}.\end{equation}
On the other hand, the rectangle spanned by $(\cos\psi,\sin\psi)$
and $\sqrt{-1}(\cos\psi,\sin\psi)$ has Holmes Thompson area $(|\cos\psi|+|\sin\psi|)^{2}$
and so is its Klain function at $P$. 
\end{example}
The double integral \prettyref{eq:di} can be transformed into an
elliptic integral,\begin{equation}
\begin{array}{lll}
I & = & 4\int_{0}^{2\pi}\sqrt{\cos^{2}\psi\sin^{2}(\theta+\psi)-\sin\theta\sin(\theta+2\psi)\sin^{2}x}dx\\
 & = & 4\cos\psi\sin(\theta+\psi)\int_{0}^{2\pi}\sqrt{1-\frac{\sin\theta\sin(\theta+2\psi)}{\cos^{2}\psi\sin^{2}(\theta+\psi)}\sin^{2}x}dx.\end{array}\label{eq:ell}\end{equation}
Assume $\theta+2\psi\leq\pi$ and let $k^{2}=\frac{\sin\theta\sin(\theta+2\psi)}{\cos^{2}\psi\sin^{2}(\theta+\psi)}$,
then by the series expansion of the incomplete elliptic integral of
the second kind,\begin{equation}
\begin{array}{lll}
I & = & 4\cos\psi\sin(\theta+\psi)\frac{\pi}{2}\sum_{m=0}^{\infty}\frac{1}{1-2m}\left(\begin{array}{c}
-\frac{1}{2}\\
m\end{array}\right)^{2}k^{2m}\\
 & = & 4\cos\psi\sin(\theta+\psi)\frac{\pi}{2}\sum_{m=0}^{\infty}\frac{1}{1-2m}\left(\begin{array}{c}
-\frac{1}{2}\\
m\end{array}\right)^{2}(\frac{\sin\theta\sin(\theta+2\psi)}{\cos^{2}\psi\sin^{2}(\theta+\psi)})^{m}\\
 & = & \pi(\sin(\theta+2\psi)+\sin\theta)\sum_{m=0}^{\infty}\frac{1}{1-2m}\left(\begin{array}{c}
-\frac{1}{2}\\
m\end{array}\right)^{2}\frac{2^{2m}\sin^{m}\theta\sin^{m}(\theta+2\psi)}{(\sin(\theta+2\psi)+\sin\theta)^{2m}}.\end{array}\label{eq:ser}\end{equation}
So we have

\begin{equation}
\begin{array}{lll}
Kl_{HT^{2}}(P) & = & \frac{1}{4\pi}I+|\cos\psi\cos(\theta+\psi)|+|\sin\psi\sin(\theta+\psi)|\\
 & = & \frac{1}{4}(\sin(\theta+2\psi)+\sin\theta)\sum_{m=0}^{\infty}\frac{1}{1-2m}\left(\begin{array}{c}
-\frac{1}{2}\\
m\end{array}\right)^{2}\frac{2^{2m}\sin^{m}\theta\sin^{m}(\theta+2\psi)}{(\sin(\theta+2\psi)+\sin\theta)^{2m}}\\
 &  & \qquad+max(|\cos\theta|,|\cos(2\psi+\theta)|).\end{array}\end{equation}

\section{The Picture on the Two Spheres }

In this section, we are going to describe the cosine transform of
the torus invariant $l^{2}$ space's singular measure in terms of
the picture of Gluck and Warner's decomposition of the Grassmannian
$Gr_{2}(\mathbb{R}^{4})$, \prettyref{eq:iso}, and Goody and Howard's
kernel characterization in \cite{GH}.

First, we can find the following correspondences under the isomorphism
\prettyref{eq:map} in terms of the spherical coordinates \prettyref{eq:p1}
and \prettyref{eq:p2}: $\mathbb{C}\times\left\{ 0\right\} $ corresponds
to $(0,0,1)$ on the first sphere $S_{+}^{2}$ and $(0,0,1))$ on
the second sphere $S_{-}^{2}$, $\left\{ 0\right\} \times\mathbb{C}$
to $(0,0,1)$ on the first sphere $S_{+}^{2}$ and $(0,0,-1))$ on
the second sphere $S_{-}^{2}$, and the torus $\mathbb{T}$ in \prettyref{eq:tor}
to the torus $S_{+}^{1}\times S_{-}^{1}$, where $S_{+}^{1}$ and
$S_{-}^{1}$ are equators of $S_{+}^{2}$ and $S_{-}^{2}$ whose $b_{1}^{+}$
and $b_{1}^{-}$ components in \prettyref{eq:1c} and \prettyref{eq:2c}
are zeros.

For any fixed point $(\xi(x,0),\eta(y,0))\in S_{+}^{2}\times S_{-}^{2}$
representing a plane $P$ in the orbit of $Gr_{2}(\mathbb{R}^{4})$
acted by torus actions, then we have\begin{equation}
\begin{array}{lll}
|\langle P,Q\rangle| & = & |\langle(\xi(x,0),\eta(y,0)),(\xi(0,\phi_{1}),\eta(0,\phi_{2}))\rangle|\\
 & = & |\cos\phi_{1}\sqrt{1-x^{2}}+\cos\phi_{2}\sqrt{1-y^{2}}|\end{array}\end{equation}
for any $(\xi(0,\phi_{1}),\eta(0,\phi_{2}))\in S_{+}^{1}\times S_{-}^{1}$
representing a plane $Q$ in torus $\mathbb{T}$ in \prettyref{eq:tor}.
So the integral in \prettyref{eq:di} appears as 

\begin{equation}
I':=\int_{0}^{2\pi}\int_{0}^{2\pi}|\cos\phi_{1}\sqrt{1-x^{2}}+\cos\phi_{2}\sqrt{1-y^{2}}|d\phi_{1}d\phi_{2}\label{eq:einteg}\end{equation}
in this picture. 

Moreover, by \prettyref{eq:p1} and \prettyref{eq:p2}, \begin{equation}
|\langle P,\mathbb{C}\times\left\{ 0\right\} \rangle|=|x+y|\label{eq:pl}\end{equation}
 and \begin{equation}
|\langle P,\left\{ 0\right\} \times\mathbb{C}\rangle|=|x-y|,\label{eq:mi}\end{equation}
and then\begin{equation}
|\langle P,\mathbb{C}\times\left\{ 0\right\} \rangle|+|\langle P,\left\{ 0\right\} \times\mathbb{C}\rangle|=2max(|x|,|y|).\label{eq:2ends}\end{equation}

So the Klain function of $HT^{2}$ can be also expressed as\begin{equation}
Kl_{HT^{2}}(P)=\frac{1}{4\pi}\int_{0}^{2\pi}\int_{0}^{2\pi}|\cos\phi_{1}\sqrt{1-x^{2}}+\cos\phi_{2}\sqrt{1-y^{2}}|d\phi_{1}d\phi_{2}+2max(|x|,|y|)\end{equation}
for any $(\xi(x,0),\eta(y,0))\in S_{+}^{2}\times S_{-}^{2}$.

The term $max(|x|,|y|)$ can be expanded as a Legendre series in ascending
total degree of the Legendre polynomial of $x$ and $y$, \begin{equation}
\begin{array}{lll}
max(|x|,|y|) & = & \frac{8}{3}+\frac{4}{15}(p_{2}(x)+p_{2}(y))-\frac{8}{105}p_{2}(x)p_{2}(y)\\
 &  & +\frac{4}{315}(p_{2}(x)p_{4}(y)+p_{4}(x)p_{2}(y))-\frac{8}{693}p_{4}(x)p_{4}(y)+\cdots\end{array}\end{equation}
Furthermore, it is not hard to rigorously show that $max(|x|,|y|)$
has only the components of $p_{m}(x)p_{n}(y)$, in which $m$ and
$n$ are even and $|m-n|=0$ or $2$, in its Legendre series expansion.
By Lemma 3.4 in \cite{GH},\begin{equation}
\int_{-1}^{1}\int_{-1}^{1}|x+y|p_{m}(x)p_{n}(y)dxdy=0\end{equation}
for $m$ and $n$ even and $|m-n|=0$ or $2$, and by substitution
\begin{equation}
\int_{-1}^{1}\int_{-1}^{1}|x+y|p_{m}(x)p_{n}(y)dxdy=\int_{-1}^{1}\int_{-1}^{1}|x-y|p_{m}(x)p_{n}(y)dxdy\end{equation}
for $m$ and $n$ even, Therefore, \begin{equation}
\int_{-1}^{1}\int_{-1}^{1}max(|x|,|y|)p_{m}(x)p_{n}(y)dxdy=0\end{equation}
 for $m$ and $n$ even and $|m-n|=0$ or $2$.

To evaluate \prettyref{eq:einteg}, we let $\sin\theta:=\sqrt{1-x^{2}}$,
$\sin\phi:=\sqrt{1-y^{2}}$, $\phi_{1}:=\alpha-\beta$ and $\phi_{2}:=\alpha+\beta$,
then

\begin{equation}
\begin{array}{lll}
I' & = & \int_{0}^{2\pi}\int_{0}^{2\pi}|\sin\theta\cos\phi_{1}+\sin\phi\cos\phi_{2}|d\phi_{1}d\phi_{2}\\
 & = & \int_{0}^{2\pi}\int_{0}^{2\pi}|\sin\theta\sin\phi_{1}+\sin\phi\sin\phi_{2}|d\phi_{1}d\phi_{2}\\
 & = & 2\int_{0}^{2\pi}\int_{0}^{2\pi}|\sin\frac{\phi+\theta}{2}\cos\frac{\phi-\theta}{2}\sin\alpha\cos\beta+\cos\frac{\phi+\theta}{2}\sin\frac{\phi-\theta}{2}\cos\alpha\sin\beta|d\alpha d\beta,\end{array}\end{equation}
which becomes the double integral \prettyref{eq:di} if we let $\phi:=2\psi+\theta$,
which can be transformed into the elliptic integral \prettyref{eq:ell}
and furthermore the series \prettyref{eq:ser}, which by back-substitution
turns out to be\begin{equation}
\begin{array}{lll}
I' & = & 2\pi(\sqrt{1-x^{2}}+\sqrt{1-y^{2}})\sum_{m=0}^{\infty}\frac{1}{1-2m}\left(\begin{array}{c}
-\frac{1}{2}\\
m\end{array}\right)^{2}\frac{2^{2m}\sqrt{(1-x^{2})^{m}(1-y^{2})^{m}}}{(\sqrt{1-x^{2}}+\sqrt{1-y^{2}})^{2m}}\\
 & = & \sum_{m=0}^{\infty}\frac{2^{2m+1}\pi}{1-2m}\left(\begin{array}{c}
-\frac{1}{2}\\
m\end{array}\right)^{2}\frac{\sqrt{(1-x^{2})^{m}(1-y^{2})^{m}}}{(\sqrt{1-x^{2}}+\sqrt{1-y^{2}})^{2m-1}}.\end{array}\end{equation}

\section{a characterization of the Klain Function}

However, we can give a characterization of the Klain function as the
following 
\begin{thm}
\label{thm:The-Klain-functionl1}The Klain function of the second
Holmes-Thompson valuation in complex $l^{1}$ space is \begin{equation}
\sum_{\begin{array}{c}
|k-l|=0\,\mbox{or}\,1\end{array}}c_{k,l\,}p_{2k}(x)p_{2l}(y)\end{equation}
 for some $c_{k,l}\in\mathbb{R}$.
\end{thm}
To prove the above theorem, let's show the following lemma first.
\begin{lem}
The cosine transform is a self-adjoint operator on the space of torus
invariant functions \begin{equation}
\mathcal{F}:=\left\{ f:Gr_{2}(\mathbb{C}^{2})\rightarrow\mathbb{R}:f\,\mathcal{\mbox{is }U}(1)\times\mathcal{U}(1)\,\mbox{-invariant}\right\} \end{equation}
\end{lem}
\begin{proof}
By Fubini's theorem,\begin{equation}
\begin{array}{lll}
\mathcal{\langle\mathcal{C_{T}}}(f),g\rangle & = & \int_{P\in Gr_{2}(\mathbb{C}^{2})}g(P)(\int_{Q\in Gr_{2}(\mathbb{C}^{2})}|\langle P,Q\rangle|f(Q)dQ)dP\\
 & = & \int_{Q\in Gr_{2}(\mathbb{C}^{2})}f(Q)(\int_{P\in Gr_{2}(\mathbb{C}^{2})}|\langle P,Q\rangle|g(P)dP)dQ\\
 & = & \mathcal{\langle}f,\mathcal{\mathcal{C_{T}}}(g)\rangle\end{array}\label{eq:selffor}\end{equation}
for any $f,g\in\mathcal{F}$.\end{proof}
\begin{rem}
The claim holds more generally for the space of all functions on $Gr_{2}(\mathbb{C}^{2})$
and furthermore distributions.
\end{rem}
Now let's prove \prettyref{thm:The-Klain-functionl1}.
\begin{proof}
First we know that $ker(\mathcal{C_{T}})$ is orthogonal to $Im(\mathcal{C_{T}})$,
because, by \prettyref{eq:selffor}, $\mathcal{\langle}f,\mathcal{\mathcal{C_{T}}}(g)\rangle=0$
for any $f\in ker(\mathcal{C_{T}})$ and any $g\in\mathcal{F}$.

On the other hand, by Lemma 3.3 in \cite{GH}, we know that $p_{m}(x)p_{n}(y)\in ker(\mathcal{C_{T}})$
for $m-n$ even and $|m-n|>2$, and since $Kl_{HT^{2}}\in Im(\mathcal{C_{T}})$,
hence \begin{equation}
\mathcal{\langle}Kl_{HT^{2}},p_{m}(x)p_{n}(y)\rangle=0\end{equation}
for $m-n$ even and $|m-n|>2$, moreover, $Kl_{HT^{2}}$ is an even
function, thus it follows that $Kl_{HT^{2}}$ has only the components
of $p_{m}(x)p_{n}(y)$ for $m$ and $n$ even and $|m-n|=0$ or $2$
in its Legendre polynomial expansion.\end{proof}
\begin{rem}
The double integral \prettyref{eq:einteg} is the cosine transform
of the delta function of $S_{+}^{1}\times S_{-}^{1}$, $\delta_{S_{+}^{1}\times S_{-}^{1}}$
on $S_{+}^{2}\times S_{-}^{2}$, which can be expanded as a series
in terms of Legendre polynomials\begin{equation}
\begin{array}{lll}
\delta_{S_{+}^{1}\times S_{-}^{1}}(x,y) & = & \delta(x)\delta(y)\\
 & = & \sum_{k=0}^{\infty}\sum_{l=0}^{\infty}\frac{p_{2k}(0)p_{2l}(0)}{\langle p_{2k}(x),p_{2k}(x)\rangle\langle p_{2l}(y),p_{2l}(y)\rangle}p_{2k}(x)p_{2l}(y)\\
 & = & \sum_{k=0}^{\infty}\sum_{l=0}^{\infty}\frac{(4k+1)(4l+1)}{4}\left(\begin{array}{c}
-\frac{1}{2}\\
k\end{array}\right)\left(\begin{array}{c}
-\frac{1}{2}\\
l\end{array}\right)p_{2k}(x)p_{2l}(y).\end{array}\label{eq:del}\end{equation}
Basically, the cosine transform annihilates the components of $p_{2k}(x)p_{2l}(y)$,
$|k-l|>1$ in \prettyref{eq:del}, and maps the components of $p_{2k}(x)p_{2l}(y)$,
$|k-l|=0$ or $1$, to those components, in other words, the cosine
transform has an invariant subspace \begin{equation}
\mathcal{G}:=\left\{ \sum_{\begin{array}{c}
|k-l|=0\,\mbox{or}\,1\end{array}}c_{k,l\,}p_{2k}(x)p_{2l}(y):c_{k,l}\in\mathbb{R}\right\} .\label{eq:g}\end{equation}

More generally, the proof for \prettyref{thm:The-Klain-functionl1}
works for any distribution on $Gr_{2}(\mathbb{C}^{2})$, so we have
the following\end{rem}
\begin{thm}[Image of cosine transform on distributions]
The image of the cosine transform on the space $D'(Gr_{2}(\mathbb{C}^{2}))$
that consists of torus invariant distributions on $Gr_{2}(\mathbb{C}^{2})$
is the space $\mathcal{G}$ in \prettyref{eq:g}. \end{thm}
\begin{acknowledgement*}
Thanks to Joe Fu for illuminating discussions and suggestions.\end{acknowledgement*}

\address{Department of Mathematics, University of Georgia, Athens, GA 30602}

\email{yliu@math.uga.edu}
\end{document}